\documentclass[letterpaper,10pt]{amsart}
\usepackage{amsmath,amssymb,amsxtra,amsfonts,amsthm,url,comment,microtype}

\usepackage{a4wide}
\usepackage{tikz} 
\usetikzlibrary{arrows,backgrounds,decorations,automata}

\theoremstyle{plain}
\newtheorem{theorem}{Theorem}
\newtheorem{lemma}{Lemma}
\newtheorem{corollary}{Corollary}
\newtheorem{proposition}{Proposition}

\theoremstyle{definition}
\newtheorem{remark}{Remark}

\newcommand{\res}[3]{\underset{{{#1}={#2}}}{\textnormal{Res}}\left\{#3\right\}}

\makeatletter
\DeclareRobustCommand\bigop[2][1]{%
  \mathop{\vphantom{\sum}\mathpalette\bigop@{{#1}{#2}}}\slimits@
}
\newcommand{\bigop@}[2]{\bigop@@#1#2}
\newcommand{\bigop@@}[3]{%
  \vcenter{%
    \sbox\z@{$#1\sum$}%
    \hbox{\resizebox{\ifx#1\displaystyle#2\fi\dimexpr\ht\z@+\dp\z@}{!}{$\m@th#3$}}%
  }%
}
\makeatother

\newcommand{\exend}{\hfill $\Diamond$}

\makeatletter
  \def\vhrulefill#1{\leavevmode\leaders\hrule\@height#1\hfill \kern\z@}
\makeatother

\newcounter{nootje}
\newcommand\noot[1]
 {\marginpar{\small\begin{minipage}{30mm}\begin{flushleft}\thenootje : #1\end{flushleft}\end{minipage}}\addtocounter{nootje}{1}}
\begin{document}

\title{Asymptotics for partitions over the Fibonacci numbers and related sequences}


\author{Michael Coons}
\author{Simon Kristensen}
\author{Mathias L.~Laursen}

\address{California State University, 400 West First Street, Chico, California 95929, USA}
\email{mjcoons@csuchico.edu}

\address{Aarhus University, Ny Munkegade 118, 8000 Aarhus C, Denmark}
\email{sik@math.au.dk, mll@math.au.dk}


\date{\today}

\begin{abstract} 
In this paper, harkening back to ideas of Hardy and Ramanujan, Mahler and de Bruijn, with the addition of more recent results on the Fibonacci Dirichlet series, we determine the asymptotic number of ways $p_F(n)$ to write an integer as the sum of non-distinct Fibonacci numbers. This appears to be the first such asymptotic result concerning non-distinct partitions over Fibonacci numbers. As well, under weak conditions, we prove analogous results for a general linear recurrences.
\end{abstract}

\maketitle

\section{Introduction}

We consider the number $p_F(n)$ of non-distinct partitions of $n$ over the Fibonacci sequence, $F_k$. Specifically, for a positive integer $n$, $p_F(n)$ is the number of solutions of \begin{equation}\label{eq:nfib} n=a_2F_2+a_3F_3+\cdots+a_kF_k+\cdots,\end{equation} in nonnegative integers $a_k$. In recent work,  Chow and Slattery \cite{CS2021} gave results on the number distinct partitions, $q_F(n)$ over the Fibonacci sequence, and noted that their work shows that neither $q_F(n)$, nor its partial sums, have a `nice' asymptotic formula. In particular, they showed that there is some oscillation in the partial sums of $q_F(n)$, and they gave bounds on these oscillations. In a similar vein, very recently, Sana \cite{CSpre} showed that there are also oscillations in the partial sums of the powers $(q_F(n))^m$ for each $m$. In this paper, with a related motivation, we determine asymptotic formulas for $p_F(n)$ and describe the oscillations that occur. 

The asymptotic theory of partitions goes back to the celebrated result Hardy and Ramanujan who, in 1918, showed that the number of ways $p(n)$ to write a positive integer as the sum of positive integers satisfies $p(n)\sim(4n\sqrt{3})^{-1}e^{\pi\sqrt{2n/3}},$ as $n\to\infty$. Here, $p(n)$ has an asymptotic with a non-oscillating main term. Mahler \cite{M1940} and de Bruijn \cite{dB1948} encountered a partition asymptotic with an oscillatory main term---they considered the number of ways $p_r(n)$ of writing $n$ as the sum of non-distinct $r$th powers, for a positive integer $r\geqslant 2$. Here, we contribute the following result.

\begin{theorem}\label{thm:mainpartitions} Let $p_F(n)$ be the number of partitions of $n$ over non-distinct Fibonacci numbers. Then, as $n\to\infty$, $$\log p_F(n)\sim \frac{(\log n)^2}{2\log\varphi}.$$ In particular, $$p_F(n)=A_F\hspace{-.1cm}\left(\frac{\log n}{\log\varphi}\right) n^{B_F(n)}(\log n)^{C_F(n)}\left(1+O\left(\frac{(\log\log n)^2}{\log n}\right)\right),$$  
where 
\begin{align*}
A_F(x)&:= \sqrt{\frac{1}{2\pi\log\varphi}}\cdot\exp\Bigg(\frac{(\log\log\varphi)^2}{2\log\varphi}+\log\log\varphi\left(\frac{c_3}{\log\varphi}-1\right)+c_2+2\gamma+\psi_1(x)\Bigg),\\
B_F(n)&:=\frac{1}{2\log\varphi}\left(\log n-2\log\log n-\frac{4\psi_0\hspace{-.1cm}\left(\frac{\log n}{\log\varphi}\right)}{\log\varphi}+2\log\log\varphi+2c_3-4\log\varphi+2\right) ,\\
C_F(n)&:= \frac{1}{2\log\varphi}\left(\log\log n-\frac{4\psi_0\hspace{-.1cm}\left(\frac{\log n}{\log\varphi}\right)}{\log\varphi}+2\log\log\varphi-2c_3+3\log\varphi+4\right),\\
c_2&:= \frac{3\gamma^2}{2\log\varphi}-\frac{\gamma_1}{\log\varphi}+\frac{\pi^2}{12\log\varphi}+2\gamma\left(\frac{\log5-\log\varphi}{2\log\varphi}\right)+\sum_{k\geqslant 1} \frac{(-1)^{k}}{k(\varphi^{2k} + (-1)^{k+1})},\mbox{ and}\\
c_3&:=\frac{1}{2}\log 5-\frac{1}{2}\log\varphi+2\gamma,
\end{align*} where $\varphi$ is the golden mean, $\gamma$ is the Euler--Mascheroni constant, $\gamma_1$ is the first Stieltjes constant, and $\psi_{0}(x)$ and $\psi_1(x)$ are explicitly computable $1$-periodic functions.
\end{theorem}

Theorem \ref{thm:mainpartitions} seems to be the first asymptotic result chracterising non-distinct partitions over Fibonacci numbers. Distinct partitions over Fibonacci numbers have received considerable attention; see \cite{CS2021} and the references therein for details. In very recent work, addressing a question of Chow and Slattery \cite{CS2021}, Kempton \cite{Kpre} has shown that $n^{-\frac{\log2}{\log\varphi}}\sum_{m\leqslant n}q_F(m)$ is $\log$-periodic. Our proof of Theorem \ref{thm:mainpartitions} gives a way to describe the related $\log$-periodic function for $p_F(n)$. 

Our main result on Fibonacci partitions followa from a result of Coons and Kirsten \cite{CK2009}, which uses a saddle-point method, which itself was inspired by the work of Nanda \cite{N1954} and Richmond \cite{R1975,R1975/76}. In particular, Richmond \cite{R1975} was not only able to give a new proof of Hardy and Ramanujan's above-mentioned result, he gave asymptotics for all of the moments of $p(n)$. The method therein, and herein, relies on the existence of invertible asymptotics for the related saddle point. In our situation, the lead order asymptotic of the saddle point is monotonic, so we have invertibility, but, additionally, we are able compute the second-order term, which is oscillatory. These terms are asymptotically close enough, so that they both contribute to the outcome of Theorem \ref{thm:mainpartitions}. We note that our results are heavily related to analytic properties of the Fibonacci zeta function (defined and discussed in more detail below); in particular, the Fibonacci zeta function is defined by a Dirichlet series which converges in the positive right half-plane. This presents added difficulties compared to the more fully-examine case of partitions related to Dirichlet series whose abscissa of convergence is strictly positive---for an interesting study on the asymptotics of partitions related to to Dirichlet series whose abscissa of convergence is strictly positive, see Debruyne and Tenenbaum~\cite{DT2020}.

This paper is organised as follows. In Sections \ref{sec:exactF} and \ref{sec:p}, we focus the proof of Theorem \ref{thm:mainpartitions}. In particular, inspired by ideas of de Bruijn \cite{dB1948}, we prove an exact formula for the generating function of $p_F(n)$ in Section \ref{sec:exactF}. We then use a saddle-point method in Section \ref{sec:p} to give the asymptotic result for $p_F(n)$. Finally, in Section \ref{sec:P}, we give the complete extension of Theorem \ref{thm:mainpartitions} to the case of a linear recurrence with dominant root. In particular, suppose that $P_k$ is a strictly increasing linearly recurrent sequence of positive integers of degree at least $2$ with $P_1 = 1$, such that the characteristic polynomial $\chi_P(x)$ of $P_k$ has a single dominant root $\beta>0$ and $P_k=\lambda\beta^k+\lambda_2\beta_2^k+\cdots+\lambda_r\beta_r^k,$ where $\lambda,\lambda_2,\ldots,\lambda_r$ are constants and $\beta_2,\ldots,\beta_r$ are the algebraic conjugates of $\beta$, then we have

\begin{theorem}\label{thm:genpart} Let $p_P(n)$ be the number of partitions of $n$ over non-distinct elements of the sequence~$P_k$. Then, as $n\to\infty$, $$p_P(n)=A_P\left(\frac{\log n}{\log\beta}\right) n^{B_P(n)}(\log n)^{C_P(n)}\left(1+O\left(\frac{(\log\log n)^2}{\log n}\right)\right),$$  
where $A_P(x)$ is the a positive $1$-periodic function satisfying \begin{align*}A_P(x)&:=\sqrt{\frac{1}{2\pi\log\beta}}\cdot\exp\Bigg(\frac{(\log\log\beta)^2}{2\log\beta}+(2\gamma-\log\lambda)\frac{\log\log\beta}{\log\beta}-\frac{1}{2}\log\log\beta+C_2+\psi_4(x)\Bigg),\\
B_P(n)&:=\frac{1}{2\log\beta}\left(\log n-2\log\log n-\frac{4\psi_3\hspace{-.1cm}\left(\frac{\log n}{\log\beta}\right)}{\log\beta}+2\log\log\beta+4\gamma -2\log\lambda-2\log\beta+2\right),\\
C_P(n)&:=\frac{1}{2\log\beta}\left(\log\log n+2-\frac{4\psi_3\hspace{-.1cm}\left(\frac{\log n}{\log\beta}\right)}{\log\beta}+2\log\log\beta-4\gamma +2\log\lambda+2\log\beta+2\right),\end{align*} where $\psi_3(x)$ and $\psi_4(x)$ are explicitly computable $1$-periodic functions, and $C_2$ is the constant defined in Proposition \ref{prop:logFPGP}.
\end{theorem}

\section{An exact formula for for the generating series of $p_F(n)$}\label{sec:exactF}

To prove Theorem \ref{thm:mainpartitions}, we study the asymptotics of the generating series
$$
F_2(x):=\sum_{n\geqslant 0}p_F(n)x^n=\prod_{k\geqslant 2}\left(1-x^{F_k}\right)^{-1}
$$
as $x \rightarrow 1^-$. Note that we are starting with $F_2=1$, and not with $F_1=1$, since we wish to avoid having two representations of $1$ in our partitions. We will necessarily need to consider a Fibonacci Dirichlet series. Navas \cite{N2001} determined most of the properties we need by considering the analytic continuation of the series $\zeta_F(z)=\sum_{k\geqslant 1}F_k^{-z}$, but here with $1$ doubly represented. It turns out this is not much of a problem. To deal with this, we merely consider the product \begin{equation}\label{eq:FtoF2}F(x)=(1-x)^{-1}F_2(x)=\prod_{k\geqslant 1}\left(1-x^{F_k}\right)^{-1},\end{equation} then translate the results back to $F_2(x)$. It is convenient to change variables, setting $x=e^{-s}$, so that we are considering the function $F(e^{-s})$, with particular interest in the asymptotics as $s \rightarrow 0^+$. 

Taking the logarithm and using the Taylor series of the logarithm near $1$,
\begin{equation}
\label{eq:asymp1}
\log F(e^{-s}) = - \sum_{k=1}^\infty \log \left(1 - e^{-sF_k}\right) = \sum_{k=0}^\infty \sum_{m=1}^\infty \frac{1}m e^{-sF_k m}.
\end{equation}
Mellin's formula for the exponential function states that for $a >0$ and $w > 0$,
$$
e^{-w} = \frac{1}{2\pi i} \int_{a-i\infty}^{a+i\infty} \Gamma(z)w^{-z}dz.
$$
Inserting this into \eqref{eq:asymp1} and interchanging integration and summation, as we may by absolute convergence, and finally rearranging the sum,
\begin{equation}
\label{eq:asymp2}
\log F(e^{-s}) = \frac{1}{2\pi i} \int_{a-i\infty}^{a+i\infty} \sum_{k=0}^\infty \sum_{m=1}^\infty \frac{1}m \Gamma(z)\left(s F_k m\right)^{-z} dz = \frac{1}{2\pi i} \int_{a-i\infty}^{a+i\infty} s^{-z} \Gamma(z) \zeta(1+z) \zeta_F(z) dz,
\end{equation}
where $a>0$, $\zeta(z+1)$ denotes the Riemann $\zeta$-function at $z+1$, and, as above $\zeta_F(z):=\sum_{k\geqslant 1}F_k^{-z}$. Note that $\zeta_F(z)$ is absolutely convergent for $\Re(z)>0$ and continuable to a meromorphic function on all of $\mathbb{C}$; this is discussed more below---see Navas \cite{N2001}.

We would like to estimate the final integral in \eqref{eq:asymp2} using Cauchy's formula; that is, we would like to move the vertical contour towards $-\infty$. Applying the functional equation of the Riemann $\zeta$-function, we find that 
$$
\lim_{n \rightarrow \infty}  \frac{1}{2\pi i} \int_{-n + \frac{1}2-i\infty}^{-n+\frac{1}2+i\infty} s^{-z} \Gamma(z) \zeta(1+z) \zeta_F(z) dz = 0,
$$
where $n$ runs over the positive integers, and, as we shall see in what follows, the vertical line avoids poles of the integrand. Also, contributions from horizontal paths of integration do not contribute in the limit, as they are moved up or down respectively. Thus, the integral in \eqref{eq:asymp2} is nothing but the sum over the residues at the poles of the integrand. In what follows, we prove  

\begin{theorem}\label{thm:Fform} The function $F_2(z)$ defined above satisfies, as $s\to0^+$, $$\log F_2(e^{-s})=\frac{(\log s)^2}{2\log\varphi}-(\log s)\left(\frac{c_3}{\log\varphi}-1\right)+c_2+2\gamma+f(s)+O(s^2).$$ where $f(s)=f(\varphi s)+O(s^2)$, $c_3:=\frac{1}{2}\log 5-\frac{1}{2}\log\varphi+2\gamma,$ and $$c_2:=\frac{3\gamma^2}{2\log\varphi}-\frac{\gamma_1}{\log\varphi}+\frac{\pi^2}{12\log\varphi}+2\gamma\left(\frac{\log5-\log\varphi}{2\log\varphi}\right)+\sum_{k\geqslant 1} \frac{(-1)^{k}}{k(\varphi^{2k} + (-1)^{k+1})}.$$
\end{theorem}

The proof of Theorem \ref{thm:Fform} will come as a direct application of five propositions, each having to do with contributions coming from certain singularities of $s^{-z} \Gamma(z) \zeta(1+z) \zeta_F(z)$, and one  lemma. To this end, note that for $s > 0$, the function $z \mapsto s^{-z}$ can be expanded in a Taylor series as
\begin{equation}
\label{eq:sto-z}
s^{-z} = 1 - z \log s + \tfrac{1}{2} (\log s)^2 z^2 + O(z^3),
\end{equation}
which converges for all $z \in \mathbb{C}$. The $\Gamma$-function has simple poles at non-positive integers, where the residue at $-n$ is $$\res{s}{-n}{\Gamma(s)}=\frac{(-1)^n}{n!}.$$
The function $\zeta(z+1)$ has a simple pole with residue $1$ at $z=0$, and the trivial zeros of $\zeta(z+1)$ at all points in $-2\mathbb{N}-1$ will cancel the corresponding poles of $\Gamma$.

The function $\zeta_F(z)$ is the most mysterious of the functions we will consider here, but much is known due to work of Navas \cite{N2001}. We will use several of these properties, which we have gathered into the following proposition.

\begin{proposition}[Navas, 2001]\label{prop:navas1} The Dirichlet series $\zeta_F(z)$ can be continued analytically to a meromorphic function on all of $\mathbb{C}$, still called $\zeta_F(z)$, whose singularities are simple poles at $s=s(n,k)=-2k+\tfrac{\pi i(2n+k)}{\log\varphi},$ for $n,k\in\mathbb{Z}$ and $k\geqslant 0$, with $$\res{z}{s}{\zeta_F(z)}=\frac{(-1)^k5^{s/2}{-s\choose k}}{\log\varphi}.$$ Moreover, we have that $\zeta_F(-(4m+2))=0$ for each $m\in\mathbb{N}_0$, and $\zeta_F(-1)=-1$.
\end{proposition}

The poles of $\zeta_F(z)$ collude with the poles of $s^{-z} \Gamma(z) \zeta(1+z)$ to become poles of higher order. Combining our knowledge of the functions $s^{-z},$ $\Gamma(z)$ and $\zeta(1+z)$ with Proposition \ref{prop:navas1}, the integrand $s^{-z} \Gamma(z) \zeta(1+z) \zeta_F(z)$ has 
\begin{itemize}
\item a triple pole at $z=0$,
\item a simple pole coming from $\Gamma(z)$ at $z=-1$,
\item double poles at $z\in-4\mathbb{N}$, and
\item simple poles off the real line at $z=s(n,k)$, for $k\geqslant 0$ and $k\neq -2n$.
\end{itemize}
Note that the poles of $\Gamma(z)$ at $z=-(4m+2)$ for $m\in\mathbb{N}_0$ are cancelled by the corresponding zero of $\zeta_F(z)$. The higher order poles require more terms in the expansions of $\Gamma(z)$, $\zeta(1+z)$ and $\zeta_F(z)$. We will consider these in the order above.

To calculate the contribution of the triple pole at $z=0$, we need the first three terms of the Laurent expansions of the contributing functions, i.e., of $\Gamma(z)$, $\zeta(z+1)$ and $\zeta_F(z)$. The former two are well known, 
\begin{equation}
\label{eq:gamma_0}
\Gamma(z) = \frac{1}{z} + \gamma + \tfrac{1}{2} \left(\gamma^2 + \frac{\pi^2}{6}\right)z + O(z^2),
\end{equation}
where $\gamma$ is the Euler--Mascheroni constant, and
\begin{equation}
\label{eq:zeta_0}
\zeta(1+z) = \frac{1}z + \gamma - \gamma_1 z + O(z^2),
\end{equation}
where $\gamma_1$ is the first Stieltjes constant,
$$
\gamma_1 = \lim_{m \rightarrow \infty}  \left(\sum_{k=1}^m \frac{\log k}{k}  - \tfrac{1}{2} (\log m)^2\right).
$$
For $\zeta_F(z)$, results beyond the residue are not available in literature, so we present them here.

\begin{lemma}\label{lem:zetaF_0} Near $z=0$, we have $$\zeta_F(z) = \frac{1}{\log\varphi}\cdot\frac{1}{z} + \frac{\log 5 - \log\varphi}{2\log\varphi} + \left(\frac{\log\varphi-3\log 5}{12} + \frac{(\log 5)^2}{8\log\varphi} + c_1 \right)z+O(z^2),$$ where $c_1:= \sum_{k\geqslant 1} \frac{(-1)^{k}}{k(\varphi^{2k} + (-1)^{k+1})}\approx -0.20436188.$
\end{lemma}

\begin{proof} The first two terms were found in \cite{N2001}. For clarity, we repeat the same process here. As a first step, noting $\bar\varphi = -1/\varphi$ and $F_n=(\varphi^n-\bar{\varphi}^n)/\sqrt{5}$, we have \begin{equation}\label{eq:phi(s) - Mathias}\zeta_F(z) = \sum_{n\geqslant 1}\frac{1}{F_n^z}  
= 5^{z/2} \sum_{n\geqslant 1} \frac{1}{(\varphi^n - \bar{\varphi}^n)^z}.\end{equation}
	We recover the Taylor expansion of $5^{z/2}$ from that of the exponential function, \begin{equation}\label{eq:sqrt(5)^s - Mathias} {5}^{z/2} = e^{\log(5)z/2} = \sum_{k\geqslant 0} \frac{(\log 5)^k}{2^k k!} z^k=1+\frac{\log 5}{2} z+\frac{(\log 5)^2}{8} z^2+O(z^3).\end{equation}
	For the series, we start with an application of the binomial theorem to give
\begin{equation*}
\sum_{n\geqslant 1} \frac{1}{(\varphi^n - \bar{\varphi}^n)^z} 
=\sum_{n\geqslant 1} \frac{1}{\varphi^{nz}} \left(1 + (-1)^{n+1} \varphi^{-2n}\right)^{-z}
=\sum_{n\geqslant 1} \frac{1}{\varphi^{nz}} \sum_{k\geqslant 0} \binom{-z}{k}(-1)^{k(n+1)} \varphi^{-2nk}.
\end{equation*}
Since the double sum is absolutely convergent, as argued in \cite{N2001}, we swap the order of summation and recognise a geometric series to give
\begin{align*}
\sum_{n\geqslant 1} \frac{1}{(\varphi^n - \bar{\varphi}^n)^z} 
&= \sum_{k\geqslant 0} \binom{-z}{k} (-1)^k \sum_{n\geqslant 1} \left((-1)^{k} \varphi^{-(z+2k)}\right)^n\\
&= \sum_{k\geqslant 0} \binom{z+k-1}{k} \frac{(-1)^{k}\varphi^{-(z+2k)}}{1- (-1)^{k}\varphi^{-(z+2k)}}\\
&= \sum_{k\geqslant 0} \frac{\Gamma(z+k)}{\Gamma(z)\Gamma(k+1)} \cdot \frac{(-1)^{k}}{\varphi^{z+2k}+ (-1)^{k+1}}\\
&= \frac{1}{\varphi^{z} -1} + \sum_{k\geqslant 1} \frac{\Gamma(z+k)}{\Gamma(z)k!} \cdot \frac{(-1)^k}{\varphi^{z+2k} + (-1)^{k+1}}.
\end{align*}
The first term is particularly cumbersome, but a few applications of L'H\^{o}pital's rule gives
\begin{align*}
\frac{1}{\varphi^z-1} = \frac{1}{\log\varphi}\cdot\frac{1}{z} -\frac{1}{2}  + \frac{\log\varphi}{12}z + O(z^2).\end{align*}

For the remaining summands, by $\tfrac{1}{\Gamma(z)} = z + O(z^2)$, $\Gamma(z+k) = \Gamma(k) + O(z)= (k-1)! + O(z)$, and $$\frac{(-1)^k}{\varphi^{z+2k} + (-1)^{k+1}} = \frac{(-1)^k}{\varphi^{2k} + (-1)^{k+1}} + O(z),$$	we have that
\begin{equation*}
\sum_{n\geqslant 1} \frac{1}{(\varphi^n - \bar{\varphi}^n)^z}	= \frac{1}{\log\varphi}\cdot\frac{1}{z} - \frac{1}{2} + \left(\frac{\log\varphi}{12} + c_1\right) z + O(z^2),\end{equation*} where $c_1$ is defined as in the statement of the lemma.
Thus, using \eqref{eq:sqrt(5)^s - Mathias},
\begin{align*}
\zeta_F(z) &= 5^{z/2}\left(\frac{1}{\log\varphi}\cdot\frac{1}{z} - \frac{1}{2} + \left(\frac{\log\varphi}{12} + c_1\right) z + O(z^2)\right)\\
&=\frac{1}{\log\varphi}\cdot\frac{1}{z} + \left(-\frac{1}{2}+\frac{\log 5}{2\log\varphi}\right) + \left(\frac{\log\varphi}{12} + c_1 - \frac{\log 5}{4} + \frac{(\log 5)^2}{8\log\varphi}\right)z + O(z^2)\\
&= \frac{1}{\log\varphi}\cdot\frac{1}{z} + \frac{\log 5 - \log\varphi}{2\log\varphi} + \left(\frac{\log\varphi-3\log 5}{12} + \frac{(\log 5)^2}{8\log\varphi} + c_1 \right)z+O(z^2),\end{align*} which is the desired result.
\end{proof}

\begin{proposition}
We have $$\res{z}{0}{s^{-z} \Gamma(z) \zeta(1+z) \zeta_F(z)}=\frac{(\log s)^2}{2\log\varphi}-\left(\frac{1}{2}\log 5-\frac{1}{2}\log\varphi+2\gamma\right)\frac{\log s}{\log\varphi}+c_2,$$ where, as in Theorem \ref{thm:mainpartitions}, $$c_2:=\frac{3\gamma^2}{2\log\varphi}-\frac{\gamma_1}{\log\varphi}+\frac{\pi^2}{12\log\varphi}+2\gamma\left(\frac{\log5-\log\varphi}{2\log\varphi}\right)+\sum_{k\geqslant 1} \frac{(-1)^{k}}{k(\varphi^{2k} + (-1)^{k+1})}.$$
\end{proposition}

\begin{proof} Around $z=0$, using \eqref{eq:sto-z}, \eqref{eq:gamma_0} and \eqref{eq:zeta_0}, we have 
$$s^{-z} \Gamma(z) \zeta(1+z) = \frac{1}{z^2}+\frac{2\gamma-\log s}{z}+\left(\frac{(\gamma-\log s)(3\gamma-\log s)}{2}-\gamma_1+\frac{\pi^2}{12}\right)+O(z).$$ Thus, by Lemma \ref{lem:zetaF_0}, we have
\begin{multline*}\res{z}{0}{s^{-z} \Gamma(z) \zeta(1+z)\zeta_F(z)}=\frac{1}{\log\varphi}\left(\frac{(\gamma-\log s)(3\gamma-\log s)}{2}-\gamma_1+\frac{\pi^2}{12}\right)\\+\frac{\log 5 - \log\varphi}{2\log\varphi}(2\gamma-\log s)+\left(\frac{\log\varphi-3\log 5}{12} + \frac{(\log 5)^2}{8\log\varphi} + c_1 \right).\end{multline*} Gathering  powers of $\log s$ finishes the proof.
\end{proof}

The simple pole at $z=-1$ is a straightforward calculation, using known values.

\begin{proposition} We have $$\res{z}{-1}{s^{-z} \Gamma(z) \zeta(1+z) \zeta_F(z)}=-\frac{s}{2}.$$
\end{proposition}

\begin{proof} We calculate, using that $\res{z}{-1}{\Gamma(z)}=-1$, $\zeta(0)=-1/2,$ and $\zeta_F(-1)=-1$ to give \begin{equation*}\res{z}{-1}{s^{-z} \Gamma(z) \zeta(1+z) \zeta_F(z)}=s\, \zeta(0)\, \zeta_F(-1)\cdot\res{z}{-1}{\Gamma(z)}=s\left(\frac{-1}{2}\right)(-1)(-1)=-\frac{s}{2}.\end{equation*} For the calculation of $\zeta_F(-1)$ see Navas \cite[Eq.~9]{N2001}
\end{proof}

For the double poles at $z\in-4\mathbb{N}$, we require the first two terms of the Laurent expansions of $\Gamma(z)$ and $\zeta_F(z)$ respectively around these points. For $\Gamma(z)$, the first term of the expansion is well known to be $\frac{1}{(4n)!}\cdot \frac{1}{z+4n}$. The constant term is surprisingly difficult to find in literature, but it can be easily calculated as the derivative of $(z+4n)\Gamma(z)$, evaluated at $z=-4n$. In order to accomplish this, we repeatedly apply the functional equation $z \Gamma(z) = \Gamma(z+1)$ to find that 
$$
\Gamma(z) = \frac{\Gamma(z+1)}{z} = \dots = \frac{\Gamma(z+4n+1)}{z(z+1) \cdots (z+4n)},
$$
so that we only need to evaluate the derivative
$$
\frac{d}{dz}\left\{\frac{\Gamma(z+4n+1)}{z(z+1) \cdots (z+4n-1)}\right\}
$$
at $z=-4n$. Recalling that $\Gamma(1) = 1$ and $\Gamma'(1) = -\gamma$, we  find that the constant term of the Laurent series of $\Gamma(z)$ around $z=-4n$ is
$$
\frac{d}{dz}\left\{\frac{\Gamma(z+4n+1)}{z(z+1) \cdots (z+4n-1)}\right\}\Big|_{z=-4n}=\frac{1}{(4n)!}\left(\sum_{k=1}^{4n} \frac{1}{k} - \gamma\right),
$$ so that near $z=-4n$, \begin{equation}\label{eq:Gamma4n}\Gamma(z)=\frac{1}{z+4n}\cdot\frac{1}{(4n)!}+ \frac{1}{(4n)!}\left(\sum_{k=1}^{4n} \frac{1}{k} - \gamma\right)+O(z+4n).\end{equation}
We also require the constant term of the Laurent series of $\zeta_F(z)$ around $z=-4n$. Following Navas~\cite{N2001},
\begin{align*}
	\zeta_F(z) &= 5^{z/2} \sum_{k=0}^\infty \binom{-z}{k} \frac{(-1)^k}{\varphi^{z+2k} + (-1)^{k+1}} = 5^{z/2} \sum_{k=0}^\infty \frac{\Gamma(1-z)}{\Gamma(1-z-k) k!}\cdot \frac{(-1)^k}{\varphi^{z+2k} + (-1)^{k+1}} \\
&= 5^{z/2}\frac{\Gamma(1-z)}{\Gamma(1-z-2n) (2n)!}\cdot \frac{1}{\varphi^{z+4n} -1}+5^{z/2} \sum_{\substack{k=0\\ k \neq 2n}}^\infty \frac{\Gamma(1-z)}{\Gamma(1-z-k) k!}\cdot \frac{(-1)^k}{\varphi^{z+2k} + (-1)^{k+1}}.
\end{align*}
The second term is holomorphic at $z=-4n$ and contributes to the constant term with its value, which we denote by $c_{-4n}$. For the first term, we note that 
$$
\lim_{z \rightarrow -4n} 5^{z/2} = 5^{-2n},
\qquad
\mbox{and}
\qquad
\lim_{z \rightarrow -4n}  \frac{\Gamma(1-z)}{\Gamma(1-z-2n) (2n)!}  =  \frac{(4n)!}{((2n)!)^2}.
$$
Thus, it remains to note that 
$$
\lim_{z \rightarrow -4n} \frac{d}{dz} \left\{\frac{z+4n}{\varphi^{z+4n} -1}\right\}=-\frac{1}{2},
$$
which is easily shown by differentiating and applying L'H\^opital's rule twice. Proposition \ref{prop:navas1} gives that the residue of $\zeta_F(z)$ at $z=-4n$ is 
$$
b_{-4n}:=\res{z}{-4n}{\zeta_F(z)} = \frac{5^{-2n}}{\log \varphi} \binom{4n}{2n} = \frac{5^{-2n}}{\log \varphi}\cdot \frac{(4n)!}{((2n)!)^2},
$$
so that near $z = -4n$, we have
\begin{equation}\label{eq:zetaF_-4n}
\zeta_F(z) = \frac{b_{-4n}}{z+4n} +\left( c_{-4n} - \frac{5^{-2n}}{2}\cdot  \frac{(4n)!}{((2n)!)^2}\right) + O(z+4n).
\end{equation}

\begin{proposition}\label{prop:g} We have
$$g(s):=\sum_{n=1}^\infty \res{z}{-4n}{s^{-z} \Gamma(z) \zeta(1+z) \zeta_F(z)}=\sum_{n=1}^\infty \alpha_n\, s^{4n}-\log s\sum_{n=1}^\infty\beta_n s^{4n},$$ where $\beta_n:={b_{-4n}\zeta(1-4n)}/{(4n)!}$, $$\alpha_n:=\frac{B_{2n}}{4n\cdot(4n)!}\left(b_{-4n}\left(\sum_{k=1}^{4n} \frac{1}{k} - \gamma\right) + c_{-4n} - \frac{5^{-2n}}{2}  \frac{(4n)!}{((2n)!)^2}\right)+\frac{b_{-4n}}{(4n)!}\cdot\zeta'(1-4n),$$ and $B_{2n}$ is the $2n$-th Bernoulli number. Moreover, as $s\to0^+$, we have $g(s)=O(s^2)$.
\end{proposition}

\begin{proof} The first part of the proposition follows from \eqref{eq:Gamma4n}, \eqref{eq:zetaF_-4n}, and the expansions, around $z=-4n$, $$s^{-z}=s^{4n}s^{-(z+4n)}=s^{4n}-(z+4n)\log s+O((z+4n)^2),$$ and $\zeta(1+z)=\zeta(1-4n)+\zeta'(1-4n)(z+4n)+O((z+4n)^2)$ along with the fact that $\zeta(1-4n)=B_{2n}/4n$; see Titchmarch \cite[p.~19]{T1986}.

For the second part, we start by using that $$\zeta'(1-4n)=\frac{(-1)^{k+1}(4n)!}{2n(2\pi)^{2n}}\zeta'(4n)+\frac{B_{4n}}{4n}\left(\sum_{k=1}^{4n-1}\frac{1}{k}-\gamma-\log(2\pi)\right).$$ Using the facts that $\zeta'(4n)\sim -2^{-4n}\log 2$ and $|B_{4n}|\sim\tfrac{2(4n)!}{(2\pi)^{4n}}$ with Stirling's approximation of the factorial, we have that $$\frac{b_{-4n}}{(4n)!}\cdot\zeta'(1-4n)=O\left(\frac{(4n)!}{(40\pi)^{2n}n((2n)!)^2}\right)
=O\left(\frac{1}{(10\pi)^{2n}n^{3/2}}\right)=O(1).$$ Now, we have at hand asymptotic information about all of the quantities in $\alpha_n$ except for $c_{-4n}$. It turns out that $c_{-4n}$ is uniformly bounded; more precisely $|c_{-4n}|<3$. To see this, note that for all $n\geqslant 1$, we have
\begin{equation*}
|c_{-4n}|=\left|\frac{1}{5^{2n}} \sum_{\substack{k=0\\ k \neq 2n}}^\infty {4n\choose k}\cdot \frac{(-1)^k}{\varphi^{-4n+2k} + (-1)^{k+1}}\right|
=\frac{1}{5^{2n}} \sum_{\substack{k=0\\ k \neq 2n}}^{4n} {4n\choose k}\cdot \varphi^{4n-2k}\left|\frac{1}{1-\varphi^{4n-2k}}\right|
\end{equation*} since ${4n\choose k}=0$ for $k>4n$. Now the value of $\left|\tfrac{1}{1-\varphi^{4n-2k}}\right|$ is maximal when $k=2n-1$ (recall, $k\neq 2n$), and there, it is approximately $2.6180$, which yields
\begin{equation*}
|c_{-4n}|<\frac{3}{5^{2n}} \sum_{\substack{k=0\\ k \neq 2n}}^{4n} {4n\choose k}\cdot \varphi^{4n-2k}\leqslant\frac{3}{5^{2n}} \sum_{\substack{k=0}}^{4n} {4n\choose k}\cdot \varphi^{4n-k}\left(\frac{1}{\varphi}\right)^{k}=\frac{3}{5^{2n}}\left(\varphi+\frac{1}{\varphi}\right)^{4n}=3,
\end{equation*} which shows that $c_{-4n}$ is uniformly bounded. With this in hand, we will now use the fact that $|B_{2n}|=\tfrac{2(2n)!}{(2\pi)^{2n}}\zeta(2n)$ and $\zeta(2n)\sim 1$ as $n\to\infty$, to finish our proof. To this end, using the definition of $b_{-4n}$ and Stirling's approximation, we have
\begin{align*}
|\alpha_n|&<\frac{|B_{2n}|}{4n\cdot(4n)!}\left(\frac{5^{-2n}}{\log \varphi}\cdot \frac{(4n)!}{((2n)!)^2}\left(\sum_{k=1}^{4n} \frac{1}{k} - \gamma\right) + 3 + \frac{5^{-2n}}{2}  \frac{(4n)!}{((2n)!)^2}\right)+O(1)\\
&=\zeta(2n)\frac{2(2n)!}{4n\cdot(2\pi)^{2n}\cdot(4n)!}\left(\frac{5^{-2n}}{\log \varphi}\cdot \frac{(4n)!}{((2n)!)^2}\left(\sum_{k=1}^{4n} \frac{1}{k} - \gamma\right) + 3 + \frac{5^{-2n}}{2}  \frac{(4n)!}{((2n)!)^2}\right)+O(1)\\
&\sim\frac{2(2n)!}{4n\cdot(2\pi)^{2n}\cdot(4n)!}\left(\frac{5^{-2n}}{\log \varphi}\cdot \frac{(4n)!}{((2n)!)^2}\log(4n) + 3 + \frac{5^{-2n}}{2}  \frac{(4n)!}{((2n)!)^2}\right)+O(1)\\
&=\frac{2}{4n\cdot(2\pi)^{2n}}\left(\frac{5^{-2n}}{\log \varphi}\cdot \frac{\log(4n)}{(2n)!} + \frac{3\cdot(2n)!}{(4n)!} + \frac{5^{-2n}}{2(2n)!}\right)+O(1)\\
&\sim \frac{2}{4n\cdot(2\pi)^{2n}}\cdot\frac{5^{-2n}}{\log \varphi}\cdot \frac{\log(4n)}{(2n)!}+O(1)\sim \frac{1}{4\sqrt{\pi}\log{\varphi}}\cdot\frac{\log n}{\sqrt{n}}\cdot\left(\frac{e}{20\pi n}\right)^{2n}+O(1)=O(1).
\end{align*} Similarly, we have that $$|\beta_n|\sim\frac{1}{4\sqrt{\pi}\log\varphi}\cdot\frac{1}{\sqrt{n}}\cdot\left(\frac{e}{20\pi n}\right)^{2n}=O(1).$$

To see $g(s)=O(s^2)$ as $s\to0^+$, we note that, using L'H\^opital's rule, we have $\lim_{s\to0^+}s^2\log s=0$. Using this, along with an application of the first part of the proposition, \begin{equation*}g(s)=\left(\alpha_1 s^4-\beta_1 s^4\log s\right)(1+O(s^4))=O(s^2).\qedhere\end{equation*}  
\end{proof}

Finally, at the simple poles of $\zeta_F(z)$ off the real line, which occur at $z=s(n,k)$ with $k\geqslant 0$ and $n\neq -k/2$, we note that $s^{-z} \Gamma(z) \zeta(1+z)$ is analytic at these points. Concerning the contributions from these singularities, we have the following result.

\begin{proposition}\label{prop:h} We have
\begin{align*} h(s):=&\, \sum_{k=0}^\infty\sum_{\substack{n=-\infty\\ n\neq -k/2}}^\infty \res{z}{s(n,k)}{s^{-z} \Gamma(z) \zeta(1+z) \zeta_F(z)}\\
=&\, \frac{1}{\log\varphi}\sum_{k=0}^\infty\sum_{\substack{n=-\infty\\ n\neq -k/2}}^\infty (-1)^k{-s(n,k)\choose k} \left(\frac{s}{\sqrt{5}}\right)^{-s(n,k)}\Gamma(s(n,k))\, \zeta(1+s(n,k)).\end{align*} Moreover, $h(s)=h(\varphi s)+O(s^2)$ as $s\to0^+$. 
\end{proposition}

\begin{proof} 
	The form $h(s)$ of these contributions is immediate using Proposition \ref{prop:navas1}. Also, since the terms in each sum over $n$ is symmetric about the real axis, $h(s)$ is real. It remains to examine the analytic properties of $h(s)$ as a function of $s$. 
	
	We first consider the $k=0$ term of $h(s)$, which we denote by $[k=0]h(s)$.
	Noting that
	\begin{align*}
		\left(\frac{\varphi s}{\sqrt{5}}\right)^{\frac{i\pi(2n)}{\log\varphi}}
		&=\cos\left(\frac{2\pi n}{\log\varphi}\log\left(\frac{\varphi \alpha}{\sqrt{5}}\right)\right)+i\sin\left(\frac{2\pi n}{\log\varphi}\log\left(\frac{\varphi \alpha}{\sqrt{5}}\right)\right)\\
		&=\cos\left(2\pi n+\frac{2\pi n}{\log\varphi}\log\left(\frac{\alpha}{\sqrt{5}}\right)\right)+i\sin\left(2\pi n+\frac{2\pi n}{\log\varphi}\log\left(\frac{\alpha}{\sqrt{5}}\right)\right)\\
		&=\cos\left(\frac{2\pi n}{\log\varphi}\log\left(\frac{\alpha}{\sqrt{5}}\right)\right)+i\sin\left(\frac{2\pi n}{\log\varphi}\log\left(\frac{\alpha}{\sqrt{5}}\right)\right)=\left(\frac{s}{\sqrt{5}}\right)^{\frac{i\pi(2n)}{\log\varphi}},
	\end{align*}
	we have that $[k=0]h(\varphi s) = [k=0]h(s)$.
	
	For $k\neq 0$, using the functional equations for $\zeta$ and $\Gamma$, since we are  examining complex values, there is a positive constant $d_{n,k}$ that is uniformly bounded such that
	\begin{align*}\big|\Gamma(s(n,k))\zeta(1+s(n,k))\big|&=\left|\frac{\zeta(-s(n,k))2^{s(n,k)}\pi^{s(n,k)+1}}{s(n,k)\sin\left(-\frac{s(n,k)\pi}{2}\right)}\right|\\
		&=\left|\frac{\pi\cdot\zeta\big(2k-i\tfrac{\pi(2n+k)}{\log\varphi}\big)}{\left(-2k+i\tfrac{\pi(2n+k)}{\log\varphi}\right)(2\pi)^{2k}\sin\left(i\tfrac{\pi^2(2n+k)}{2\log\varphi}\right)}\right|\sim d_{n,k}\cdot \frac{e^{-\frac{\pi^2k}{2\log\varphi}}}{(2\pi)^{2k}}\cdot e^{-\frac{\pi^2|n|}{\log\varphi}},
	\end{align*} where we have used that $|\sin(z)|$ is $\pi$-periodic in the $\Re(z)$, that $|e^{i\theta}|=1$ for all $\theta$, and that, as $k$ or $|n|$ (or both) grows, $\left|\sin\big(i\tfrac{\pi^2(2n+k)}{2\log\varphi}\big)\right|\sim {e^{\frac{\pi^2k}{2\log\varphi}}e^{\frac{\pi^2|n|}{\log\varphi}}}/{2}.$ 
	
	It remains to deal with the factor ${-s(n,k)\choose k}$. To this end, note that \begin{multline*}\left|{-s(n,k)\choose k}\right|=\left|\frac{\Gamma(1-s(n,k))}{k!\, \Gamma(1-s(n,k)-k)}\right|=\frac{1}{k!}\left(\prod_{j=k+1}^{2k}\left(j^2+\big(\tfrac{\pi(2n+k)}{\log\varphi}\big)^2\right)\right)^{1/2}\\
		= \frac{(kn)^{k}}{k!}\left(\prod_{j=k+1}^{2k}\left(\left(\frac{j}{nk}\right)^2+\left(\tfrac{\pi\left(\frac{2}{k}+\frac{1}{n}\right)}{\log\varphi}\right)^2\right)\right)^{1/2}\leqslant \frac{(kn)^{k}}{k!}\left(\left(\frac{2}{n}\right)^2+\left(\tfrac{\pi\left(\frac{2}{k}+\frac{1}{n}\right)}{\log\varphi}\right)^2\right)^{k/2},\end{multline*} and, independent of $k$,
	$$\left(\left(\frac{2}{n}\right)^2+\left(\tfrac{\pi\left(\frac{2}{k}+\frac{1}{n}\right)}{\log\varphi}\right)^2\right)^{k/2}\xrightarrow{n\to\infty}\left(\frac{2\pi}{k\log\varphi}\right)^{k}.$$ Thus, there is a $d>0$, independent from $n$ and $k$, such that any term of $h(s)$ with $k\neq 0$, satisfies, \begin{align*}\big|[k\neq 0]h(s)\big|&\leqslant d\sum_{\substack{n=-\infty\\ n\neq -k/2}}^\infty \left(\frac{s}{\sqrt{5}}\right)^{2k}\frac{(kn)^{k}}{k!}\left(\frac{2\pi}{k\log\varphi}\right)^{k}\cdot d_{n,k}\cdot \frac{e^{-\frac{\pi^2k}{2\log\varphi}}}{(2\pi)^{2k}}\cdot e^{-\frac{\pi^2|n|}{\log\varphi}}\\
		&= d\left(\frac{s}{\sqrt{5}}\right)^{2k}\frac{1}{k!}\left(\frac{1}{\log\varphi}\right)^{k}\frac{e^{-\frac{\pi^2k}{2\log\varphi}}}{(2\pi)^{k}}\sum_{\substack{n=-\infty\\ n\neq -k/2}}^\infty d_{n,k}\cdot n^{k}\cdot  e^{-\frac{\pi^2|n|}{\log\varphi}},
	\end{align*} where $d_{n,k}$ here is different from above, but still a uniformly bounded positive constant. So, there is a positive constant $d$, also different from above, but still independent of $n$ and $k$, such that 
	\begin{align*}
		\nonumber\big|[k\neq 0]h(s)\big|&\leqslant d\left(\frac{s^2}{10\pi\cdot\log\varphi\cdot e^{\frac{\pi^2}{2\log\varphi}}}\right)^{k}\frac{1}{k!} \sum_{n=0}^\infty n^{k}  e^{-\frac{\pi^2|n|}{\log\varphi}}\\
		\nonumber&=d\left(\frac{s^2}{10\pi\cdot\log\varphi\cdot e^{\frac{\pi^2}{2\log\varphi}}}\right)^{k}\frac{1}{k!}\cdot\frac{e^{-\frac{\pi^2}{2\log\varphi}} A_{k}\big(e^{-\frac{\pi^2}{2\log\varphi}}\big)}{\big(1-e^{-\frac{\pi^2}{2\log\varphi}}\big)^{k+1}}\\
		\nonumber&=d\left(\frac{s^2}{10\pi\cdot\log\varphi\cdot \big(e^{\frac{\pi^2}{2\log\varphi}}-1\big)}\right)^{k}\frac{1}{k!}\cdot\frac{A_{k}\big(e^{-\frac{\pi^2}{2\log\varphi}}\big)}{\big(e^{\frac{\pi^2}{2\log\varphi}}-1\big)}\\
		&<\frac{d}{\big(e^{\frac{\pi^2}{2\log\varphi}}-1\big)}\left(\frac{s^2}{10\pi\cdot\log\varphi\cdot \big(e^{\frac{\pi^2}{2\log\varphi}}-1\big)}\right)^{k},
	\end{align*} where, for $k\geqslant 1$, we have used that $\sum_{n\geqslant 0}n^kx^n={x\, A_k(x)}{(1-x)^{-(k+1)}},$ where $A_k(x)\in\mathbb{Z}_{>0}[x]$ is the $k$-th Eulerian Polynomial, which satisfies $\deg A_k(x)=k-1$ and $A_k(1)=k!$. 
	Hence, the terms $[k\neq 0]h(s)$ contribute $O(s^{2})$ collectively, and the lemma follows.\qedhere
\end{proof}
\begin{lemma}\label{lem:1-es} As $s\to0^+$, we have $$\log(1-e^{-s})=\log s-2\gamma-\frac{s}{2}+O(s^2).$$
\end{lemma}

\begin{proof} We follow the method above, writing $$-\log(1-e^{-s})=\frac{1}{2\pi i}\int_{a-i\infty}^{a+i\infty} s^{-z}\Gamma(z)\zeta(1+z)dz,$$ where, for now, $a>0$. We use the expansions of $\Gamma(z)$ and $\zeta(1+z)$ around $z=0$ from above, with the fact that the integrand has a simple pole at $z=-1$ coming from $\Gamma(z)$ to get that, as $s\to 0^+$, $$-\log(1-e^{-s})=-\log s+2\gamma-s\zeta(0)+O(s^2)$$ which, since $\zeta(0)=-1/2$, when multiplied by $-1$, yields the desired result.
\end{proof}

Using the relationship in \eqref{eq:FtoF2} and combining this lemma with the previous four propositions proves Theorem \ref{thm:Fform}.

\section{Fibonacci partitions $p_F(n)$ via the saddle point method}\label{sec:p}

In this section, we prove our main result using a saddle point method. To achieve this, we must determine the behaviour as $s\to0^+$ of each of the pieces in the expansion of $\log F(e^{-s})$. 

We begin in the same way as Hardy and Ramanujan \cite{HR1918}, using Cauchy's integral formula, but diverge from their argument almost immediately. We have  $$p_F(n)=\frac{1}{2\pi i}\int_{C_\varepsilon} \frac{F_2(z)}{z^{m+1}}dz=\frac{1}{2\pi i}\int_{C_\varepsilon} \frac{e^{n(-\log z+\frac{1}{n}\log F_2(z))}}{z}dz=\frac{1}{2\pi i}\int_{s.p.}e^{n(s+\frac{1}{n}\log F_2(e^{-s}))}ds,$$ where $C_\varepsilon$ indicates a positively oriented circle of radius $\varepsilon\in(0,1)$ and  $s.p.$ indicates a path that goes through the saddle point  $s=\alpha$ of the integrand, that is, the point $s=\alpha$ which is the solution of the equation \begin{equation}\label{eq:saddle}\frac{d}{ds}\left(s+\frac{1}{n}\log F_2(e^{-s})\right)=0.\end{equation} Our main result on Fibonacci partitions will follow from a result of Coons and Kirsten \cite{CK2009}, which itself was inspired by the work of Nanda \cite{N1954} and Richmond \cite{R1975,R1975/76}. 

\begin{theorem}[Coons and Kirsten, 2009]\label{thm:CK2009}
If $\Lambda(x)=\prod_{k\geqslant 1}\left(1- x^{\lambda_k}\right)^{- 1}$ generates a sequence $p_\lambda(n)$, and $s=\alpha$ is the solution of \eqref{eq:saddle} with $F_2(e^{-s})$ replaced by $\Lambda(e^{-s})$, then, as $n$ tends to infinity,  $$p_\lambda(n)=\frac{e^{n\alpha}\Lambda(e^{-\alpha})}{\sqrt{2\pi}}\left(\sqrt{\frac{1}{-\left.\frac{dn}{ds}\right|_{s=\alpha}}}+O\left(\frac{1}{n^{3/2}}\right)\right).$$ Here, $\alpha$ must be thought of as being replaced by its large-$n$ asymptotic expansion so that the asymptotic of $p_\lambda(n)$ represents a large-$n$ asymptotic.
\end{theorem}

\begin{remark}The proof of Theorem \ref{thm:CK2009} was accomplished by iteratively applying an asymptotic result on exponential integrals, which  can be found in the book of Olver \cite[p.~127, Theorem 7.1]{O1974}. Note that Coons and Kirsten \cite{CK2009} use the notation $t_\Lambda^0(n)$ for our definition of $p_\lambda(n)$ above. \exend
\end{remark}
 
\begin{remark} The statement ``$\alpha$ must be thought of as being replaced by its large-$n$ asymptotic expansion'' may seem a bit cumbersome, but, here, the point is that the solution of \eqref{eq:saddle} gives $n$ as a function of the saddle point $\alpha$ as an (asymptotically) monotonic function, so it is invertible; that is, there is a well-defined asymptotic for the saddle point $\alpha$ in terms $n$---this is precisely the method we employ.\end{remark} 

To apply this Theorem \ref{thm:CK2009}, we must first determine the saddle point for large values of $n$. From \eqref{eq:saddle}, we have that \begin{equation}\label{eq:nint}n=-\left.\frac{d}{ds}\log F_2(e^{-s})\right|_{s=\alpha}.\end{equation} We will use this combined with Theorem \ref{thm:Fform} to prove the following result.

\begin{lemma}\label{lem:nalpha} There exists a function $h_0(s)$ satisfying $h_0(s)=h_0(\varphi s)$ such that for sufficiently large $n$, or, equivalently, for sufficiently small $\alpha>0$, $$n=\frac{-\log(\alpha)}{\alpha\log\varphi}+\frac{h_0(\alpha)}{\alpha}+O\left(\alpha\right).$$ 
\end{lemma}

\begin{proof} Let $f(s):=g(s)+h(s)$, where $g(s)$ and $h(s)$ are as in Propositions \ref{prop:g} and \ref{prop:h}, respectively. As $s\to0^+$, using Proposition \ref{prop:g} and one application of L'H\^opital's rule gives $g(s)=O(s^2)$ (in fact, one gets that $g(s)=O(s^{4-\varepsilon})$ for any fixed small positive $\varepsilon$, but only $O(s^2)$ is necessary after later comparison with the asymptotics for $h(s)$). Now, collectively, the sum of all of the terms with $k>0$ in the formula for $h(s)$ in Proposition \ref{prop:h} go to zero as $s\to0^+$, since $\Re(-s(n,k))=2k>0$. Thus, as $s=\alpha\to0^+$, separating out the $k=0$ term of the sum, denoting it $[k=0]h(s)$, and noting that $s(n,0)=\frac{2\pi i n}{\log\varphi}$, we have, $$ f(\alpha)=[k=0]h(\alpha)+O(\alpha^2)=\frac{1}{\log\varphi}\sum_{\substack{n=-\infty\\ n\neq 0}}^\infty \left(\frac{\alpha}{\sqrt{5}}\right)^{-\frac{2\pi i n}{\log\varphi}}\Gamma(\tfrac{2\pi i n}{\log\varphi})\, \zeta(1+\tfrac{2\pi i n}{\log\varphi})+O(\alpha^{2}),$$ so that $$f'(\alpha)=\frac{\sqrt{5}}{\alpha}\cdot [k=0]h(\alpha)+O(\alpha).$$ The form of $[k=0]h(s)$ immediately implies that $[k=0]h(s)=[k=0]h(\varphi s)$.

Now, we calculate \begin{align}\nonumber n&=-\left.\frac{d}{ds}\log F_2(e^{-s})\right|_{s=\alpha}\\
\nonumber&=-\left.\frac{d}{ds}\left(\frac{(\log s)^2}{2\log\varphi}-(\log s)\left(\frac{c_3}{\log\varphi}-1\right)+c_2+2\gamma+f(s)+O(s^2)\right)\right|_{s=\alpha}\\
\label{eq:na}&=-\frac{\log \alpha}{\alpha\log\varphi}-\frac{1}{\alpha}\left(\frac{c_3}{\log\varphi}-1-\sqrt{5}\cdot [k=0]h(\alpha)\right)+O\left(\alpha\right).
\end{align} setting $h_0(s):=\sqrt{5}\cdot [k=0]h(s)+1-c_3/\log\varphi$ gives the result.
\end{proof}

\begin{remark} Lemma \ref{lem:nalpha} provides the leading two terms for $n$ in terms of the saddle point $\alpha$. Here, the first term shows that the relationship is asymptotically monotonic, so that it can be inverted. The second term is oscillatory. If one follows this method and tries to apply it to the distinct Fibonacci partitions function $q_F(n)$ the resulting asymptotic is not monotonic---the leading term is oscillatory. This is precisely why this method doesn't immediately generalise to distinct Fibonacci partitions. 
\end{remark}

While the above lemma gives $n$ as a function of $\alpha$, we necessarily need $\alpha$ as a function of $n$ to apply the saddle point method. We achieve this via the Lambert $W$-function. 
        
\begin{proposition}\label{prop:alphan} There is a continuous $1$-periodic function $\psi_0(x)$, such that for sufficiently small $\alpha>0$, or, equivalently, for sufficiently large $n$, $$\alpha=\frac{W\left(e^{\psi_0\left(\frac{\log n}{\log\varphi}\right)/\log\varphi}n/\log\varphi\right)}{n\log \varphi}\left(1+O\left(\frac{\log n}{n^2}\right)\right),$$ where $W(x)$ denotes Lambert's $W$-function.
\end{proposition}

\begin{proof}
Note that the previous lemma gives that $$n=\left(\frac{-\log(\alpha)}{\alpha\log\varphi}+\frac{h_0(\alpha)}{\alpha}\right)\left(1+O\left(\frac{\alpha^2}{\log{\alpha}}\right)\right),$$ where $h_0(s)$ is fixed along any sequence $\{x\varphi^m\}_{m\geqslant 0}$. We use this property to invert the above relationship between $\alpha$ and $n$. Note that the relationship is invertible because the lead asymptotics are strictly monotonic.
Now, when one inverts $n\sim A\frac{1}{\alpha}\log \frac{1}{\alpha}+B\frac{1}{\alpha}$, one gets $\alpha\sim A\cdot W\left(e^{B/A}\frac{n}{A}\right)/n\sim A\log n/n$, where $W(x)$ is Lambert's $W$-function. Doing this along the sequences $\{x\varphi^m\}_{m\geqslant 0}$ to ensure a constant $B=B(x)$, we then reconstruct, using a fundamental interval, say $x\in[\varphi,\varphi^2]$, to get a continuous $1$-periodic function $\psi_0(x)$ such that for large $n$,  $$\alpha=\frac{W\left(e^{\psi_0\left(\frac{\log n}{\log\varphi}\right)/\log\varphi}n/\log\varphi\right)}{n\log \varphi}\left(1+O\left(\frac{\log n}{n^2}\right)\right).$$ Here, we have used the original relationship to find that $O(\alpha/\log\alpha)=O(1/n),$ and then the inverse, noting that $W(n)\sim\log n$. 
\end{proof}

In what follows, we will use Proposition \ref{prop:alphan} to give asymptotics for several functions of $\alpha$, including $n\alpha$, $\left.\frac{dn}{ds}\right|_{s=\alpha}$, $\log\alpha$ and $(\log\alpha)^2$. To give our end result, we will need varying orders of precision for the asymptotics for each of these terms. In particular, first order asymptotics of the Lambert $W$-function will not be enough to deal with 
$(\log\alpha)^2$, though they will be enough for some terms so we record them below. For convenience, we note here that the Lambert $W$-function satisfies, \begin{equation}\label{eq:W}W(x)=\log x-\log\log x+\frac{\log\log x}{\log x}+O\left(\frac{\log\log x}{(\log x)^2}\right),\end{equation} as $x\to\infty$; see, e.g., Corless et al.~\cite{C1996}.                                                                                             

\begin{corollary}\label{cor:alphafirstorder} For sufficiently small $\alpha>0$, or, equivalently, for sufficiently large $n$, $$\alpha=\frac{\log n}{n\log \varphi}\left(1+O\left(\frac{\log \log n}{\log n}\right)\right),$$ and $$\log\alpha=-\log n+\log\log n-\log\log\varphi+O\left(\frac{\log \log n}{\log n}\right).$$
\end{corollary}

\begin{proof} The first result follows directly from the fact that
 $$W(x)=\log x(1+O(\log\log x/\log x))$$ for all $x$ sufficiently large. The second follows immediately from the first.
\end{proof} 

\begin{corollary}\label{cor:dndalpha} For sufficiently small $\alpha>0$, or, equivalently, for sufficiently large $n$, $$\left.\frac{dn}{ds}\right|_{s=\alpha}=\frac{-n^2\log \varphi}{\log n}\left(1+O\left(\frac{\log \log n}{\log n}\right)\right).$$
\end{corollary}

\begin{proof} We start with Lemma \ref{lem:nalpha} in the form $n=-\log(\alpha)/(\alpha\log\varphi)+O\left({1}/{\alpha}\right),$ and take a derivative, then apply both parts of Corollary \ref{cor:alphafirstorder} to get $$\left.\frac{dn}{ds}\right|_{s=\alpha}=\frac{\log \alpha}{\alpha^2\log\varphi}+O\left(\frac{1}{\alpha^2}\right)=\frac{\log \alpha}{\alpha^2\log\varphi}\left(1+O\left(\frac{1}{\log\alpha}\right)\right)=\frac{-n^2\log \varphi}{\log n}\left(1+O\left(\frac{\log \log n}{\log n}\right)\right),$$ which finishes the proof.
\end{proof}

The final term necessary is $(\log \alpha)^2$. Here, in addition to using the full asymptotic in \eqref{eq:W}, we will use the fact that for any $y$, $\log W(y)=\log y-W(y)$.

\begin{corollary}\label{cor:log2} For sufficiently small $\alpha>0$, or, equivalently, for sufficiently large $n$, \begin{multline*}(\log\alpha)^2=\log n\left(\log n-2\log\log n-\frac{4\psi_0\hspace{-.1cm}\left(\frac{\log n}{\log\varphi}\right)}{\log\varphi}+2\log\log\varphi\right)\\ +\log\log n\left(\log\log n+2-\frac{4\psi_0\hspace{-.1cm}\left(\frac{\log n}{\log\varphi}\right)}{\log\varphi}+2\log\log\varphi\right)\\ +(\log\log\varphi)^2-\frac{2\psi_0\hspace{-.1cm}\left(\frac{\log n}{\log\varphi}\right)}{\log\varphi}+2\log\log\varphi+O\left(\frac{(\log\log n)^2}{\log n}\right).\end{multline*}
\end{corollary}
 
\begin{proof} We start with Proposition \ref{prop:alphan} and take the natural logarithm of both sides to obtain \begin{equation}\label{eq:logW2}(\log\alpha)^2=\left(\log\left(\frac{W\left(e^{\psi_0\left(\frac{\log n}{\log\varphi}\right)/\log\varphi}n/\log\varphi\right)}{n\log \varphi}\right)\right)^2+O\left(\frac{(\log n)^2}{n^2}\right),\end{equation} since $$\log\left(\frac{W\left(e^{\psi_0\left(\frac{\log n}{\log\varphi}\right)/\log\varphi}n/\log\varphi\right)}{n\log \varphi}\right)=O(\log n).$$ We now use the full force of \eqref{eq:W} to give an asymptotic for the first term in \eqref{eq:logW2}, first noting that \begin{align*}\log\left(\frac{W\left(e^{\psi_0\left(\frac{\log n}{\log\varphi}\right)/\log\varphi}n/\log\varphi\right)}{n\log \varphi}\right)&=\log\left(e^{\psi_0\left(\frac{\log n}{\log\varphi}\right)/\log\varphi}n/\log\varphi\right)\\
&\qquad\qquad-\log(n\log\varphi)-W\left(e^{\psi_0\left(\frac{\log n}{\log\varphi}\right)/\log\varphi}n/\log\varphi\right)\\
&=\frac{\psi_0\hspace{-.1cm}\left(\frac{\log n}{\log\varphi}\right)}{\log\varphi}-W\left(e^{\psi_0\left(\frac{\log n}{\log\varphi}\right)/\log\varphi}n/\log\varphi\right).
\end{align*} Thus, \begin{multline}\label{logW2}\left(\log\left(\frac{W\left(e^{\psi_0\left(\frac{\log n}{\log\varphi}\right)/\log\varphi}n/\log\varphi\right)}{n\log \varphi}\right)\right)^2=\left(\frac{\psi_0\hspace{-.1cm}\left(\frac{\log n}{\log\varphi}\right)}{\log\varphi}\right)^2\\ -\frac{2\psi_0\hspace{-.1cm}\left(\frac{\log n}{\log\varphi}\right)}{\log\varphi}W\left(e^{\psi_0\left(\frac{\log n}{\log\varphi}\right)/\log\varphi}n/\log\varphi\right)+W\left(e^{\psi_0\left(\frac{\log n}{\log\varphi}\right)/\log\varphi}n/\log\varphi\right)^2.\end{multline} For the middle term we use the first three terms of the asymptotic of $W$ in \eqref{eq:W} and, for the square, we will use the square of all of \eqref{eq:W}, which is, as $x\to \infty$,  $$W(x)^2=(\log x)^2+(\log\log x)^2-2\log x\log\log x+2\log\log x+O\left(\frac{(\log\log x)^2}{\log x}\right).$$ To this end, we to determine strong estimates for the asymptotics of $\log x$, $\log\log x$ and $\log x \log\log x$ with $x=e^{\psi_0\left(\frac{\log n}{\log\varphi}\right)/\log\varphi}n/\log\varphi.$ Here, we have \begin{align}\label{logna}\log\left(e^{\psi_0\left(\frac{\log n}{\log\varphi}\right)/\log\varphi}n/\log\varphi\right)&=\log n+\frac{\psi_0\hspace{-.1cm}\left(\frac{\log n}{\log\varphi}\right)}{\log\varphi}-\log\log\varphi\\
\nonumber&=\log n\left(1+\frac{\psi_0\hspace{-.1cm}\left(\frac{\log n}{\log\varphi}\right)/\log\varphi-\log\log\varphi}{\log n}\right),\end{align} and so, using that as $y\to 0$, $\log(1+y)=y-y^2/2+O(y^3)$, we have \begin{equation}\label{loglogna}\log\log\left(e^{\psi_0\left(\frac{\log n}{\log\varphi}\right)/\log\varphi}n/\log\varphi\right)=\log\log n+\frac{\psi_0\hspace{-.1cm}\left(\frac{\log n}{\log\varphi}\right)/\log\varphi-\log\log\varphi}{\log n}+O\left(\frac{1}{(\log n)^2}\right).\end{equation} So, using the above asymptotics, \begin{equation}\label{W} W\left(e^{\psi_0\left(\frac{\log n}{\log\varphi}\right)/\log\varphi}n/\log\varphi\right)=\log n+\log\log n+\frac{\psi_0\hspace{-.1cm}\left(\frac{\log n}{\log\varphi}\right)}{\log\varphi}-\log\log\varphi +O\left(\frac{\log\log n}{\log n}\right),\end{equation} and \begin{multline}\label{W2} W\left(e^{\psi_0\left(\frac{\log n}{\log\varphi}\right)/\log\varphi}n/\log\varphi\right)^2=(\log n)^2-2\log n\log\log n\\ -2\log n\left(\frac{\psi_0\hspace{-.1cm}\left(\frac{\log n}{\log\varphi}\right)}{\log\varphi}-\log\log\varphi\right)
+(\log\log n)^2+2\log\log n\left(1+\log\log\varphi-\frac{\psi_0\hspace{-.1cm}\left(\frac{\log n}{\log\varphi}\right)}{\log\varphi}\right)\\+\left(\frac{\psi_0\hspace{-.1cm}\left(\frac{\log n}{\log\varphi}\right)}{\log\varphi}-\log\log\varphi\right)^2
-2\left(\frac{\psi_0\hspace{-.1cm}\left(\frac{\log n}{\log\varphi}\right)}{\log\varphi}-\log\log\varphi\right)+O\left(\frac{(\log\log n)^2}{\log n}\right).\end{multline} Combining \eqref{eq:logW2}, \eqref{logW2}, \eqref{W} and \eqref{W2} gives the result.
\end{proof}

\begin{proposition}\label{prop:f} Let $f(s):=g(s)+h(s)$, where $g(s)$ and $h(s)$ are as in Propositions \ref{prop:g} and \ref{prop:h}, respectively. Then, as $s\to0^+$, the function $f(s)$ satisfies $f(s)=f(\varphi s)+O(s^2)$. That is, for sufficiently small $\alpha>0$, or equivalently, for sufficiently large $n$, $$f(\alpha)=\psi_1\left(\frac{\log n}{\log\varphi}\right)+O\left(\frac{\log\log n}{\log n}\right),$$ for some $1$-periodic function $\psi_1(x)$, and any $\varepsilon>0$.
\end{proposition}

\begin{proof} We proceed as in the proof of Lemma \ref{lem:nalpha}, and, as in that proof, separating out the $k=0$ term of the sum for $f(\alpha)$ and denoting it $[k=0]h(s)$, as $s=\alpha\to0^+$, \begin{align}\nonumber f(\alpha)=[k=0]h(\alpha)+O(\alpha^2)&=\frac{1}{\log\varphi}\sum_{\substack{n=-\infty\\ n\neq 0}}^\infty \left(\frac{\alpha}{\sqrt{5}}\right)^{-\frac{2\pi i n}{\log\varphi}}\Gamma(\tfrac{2\pi i n}{\log\varphi})\, \zeta(1+\tfrac{2\pi i n}{\log\varphi})+O(\alpha^{2})\\
\nonumber&=\frac{2}{\log\varphi}\sum_{n=1}^\infty \Re\left(\left(\frac{\alpha}{\sqrt{5}}\right)^{-\frac{2\pi i n}{\log\varphi}}\Gamma(\tfrac{2\pi i n}{\log\varphi})\, \zeta(1+\tfrac{2\pi i n}{\log\varphi})\right)+O(\alpha^{2})\\
\label{eq:hsperiod}&=\frac{2}{\log\varphi}\sum_{n=1}^\infty \Re\left(\left(\cos\left(\tfrac{2\pi  n}{\log\varphi}\log\left(\tfrac{\alpha}{\sqrt{5}}\right)\right)-i\sin\left(\tfrac{2\pi  n}{\log\varphi}\log\left(\tfrac{\alpha}{\sqrt{5}}\right)\right)\right)\right.\\
\nonumber&\qquad\qquad\qquad\qquad\left.\times\Gamma(\tfrac{2\pi i n}{\log\varphi})\, \zeta(1+\tfrac{2\pi i n}{\log\varphi})\right)+O(\alpha^{2}).
\end{align} Since both the sine and cosine functions are $2\pi$-periodic, \eqref{eq:hsperiod} gives that $f(s)=f(\varphi s)+O(s^2)$ as $s\to0^+$. Applying Corollary \ref{cor:alphafirstorder}, we obtain $$\frac{2\pi  j}{\log\varphi}\log\left(\tfrac{\alpha}{\sqrt{5}}\right)=-2\pi  j\left(\frac{\log n}{{\log\varphi}}\right)+O\left(\frac{\log\log n}{\log n}\right).$$ Finally, we note that $O(\alpha^{2})=O((\log n)^2/n^2)=O\left(\log\log n/\log n\right)$ and set $\psi_1(s)=[k=0]h(s)$ to finish the proof.
\end{proof}

We now have all of the elements to continue with our proof Theorem \ref{thm:mainpartitions} in the case of non-distinct partitions of $n$.

\begin{proof}[Proof of Theorem \ref{thm:mainpartitions}] We evaluate the pieces of the asymptotic in Theorem \ref{thm:CK2009}. Towards this end, Proposition \ref{prop:alphan} combined with \eqref{W} gives, \begin{align*} {n\alpha}&={\frac{1}{\log\varphi}W\left(e^{\psi_0\left(\frac{\log n}{\log\varphi}\right)/\log\varphi}n/\log\varphi\right)\left(1+O\left(\frac{\log n}{n^2}\right)\right)}\\
&=\frac{\log n}{\log\varphi}+\frac{\log\log n}{\log\varphi}+\frac{\psi_0\hspace{-.1cm}\left(\frac{\log n}{\log\varphi}\right)}{(\log\varphi)^2}-\frac{\log\log\varphi}{\log\varphi} +O\left(\frac{\log\log n}{\log n}\right),\end{align*} so that \begin{equation}\label{eq:T1}e^{n\alpha}=e^{\frac{1}{(\log\varphi)^2}{\psi_0\left(\frac{\log n}{\log\varphi}\right)}-\frac{\log\log\varphi}{\log\varphi}}n^{1/\log\varphi}(\log n)^{1/\log\varphi}\left(1+O\left(\frac{\log\log n}{\log n}\right)\right)\end{equation} Corollary \ref{cor:dndalpha} gives \begin{equation}\label{eq:T2}\left(\sqrt{\frac{1}{-\left.\frac{dn}{ds}\right|_{s=\alpha}}}+O\left(\frac{1}{n^{3/2}}\right)\right)=\frac{1}{n}\left(\frac{\log n}{\log\varphi}\right)^{1/2}\left(1+O\left(\frac{\log\log n}{\log n}\right)\right).\end{equation} It remains to determine the small $\alpha$, and hence large $n$, asymptotics of $F_2(e^{-\alpha})$. To this end, we combine Theorem \ref{thm:Fform} with Corollaries \ref{cor:alphafirstorder} and \ref{cor:log2}, and Proposition \ref{prop:f} to obtain, as $n\to\infty$, \begin{align*}\log F_2(e^{-\alpha})&=\frac{(\log \alpha)^2}{2\log\varphi}-(\log \alpha)\left(\frac{c_3}{\log\varphi}-1\right)+c_2+2\gamma+f(\alpha)+O(\alpha^2)\\
&=\frac{\log n}{2\log\varphi}\left(\log n-2\log\log n-\frac{4\psi_0\hspace{-.1cm}\left(\frac{\log n}{\log\varphi}\right)}{\log\varphi}+2\log\log\varphi+2c_3-2\log\varphi\right)\\
&\qquad+\frac{\log\log n}{2\log\varphi}\left(\log\log n+2-\frac{4\psi_0\hspace{-.1cm}\left(\frac{\log n}{\log\varphi}\right)}{\log\varphi}+2\log\log\varphi-2c_3+2\log\varphi\right)\\
&\qquad\qquad+\frac{(\log\log\varphi)^2}{2\log\varphi}-\frac{2\psi_0\hspace{-.1cm}\left(\frac{\log n}{\log\varphi}\right)}{2(\log\varphi)^2}+\frac{\log\log\varphi}{\log\varphi}+\log\log\varphi\left(\frac{c_3}{\log\varphi}-1\right)\\
&\qquad\qquad\qquad+c_2+2\gamma+\psi_1\left(\frac{\log n}{\log\varphi}\right)+O\left(\frac{(\log\log n)^2}{\log n}\right).\end{align*} Thus, we have \begin{equation*}F_2(e^{-\alpha})=\psi_2(n)n^{a(n)}(\log n)^{b(n)}\left(1+O\left(\frac{(\log\log n)^2}{\log n}\right)\right),\end{equation*} where $\psi_2(n)$ is the strictly positive bounded (above and below) function \begin{align*}\psi_2\left(n\right)&:=\exp\Bigg(\frac{(\log\log\varphi)^2}{2\log\varphi}-\frac{2\psi_0\hspace{-.1cm}\left(\frac{\log n}{\log\varphi}\right)}{2(\log\varphi)^2}+\frac{\log\log\varphi}{\log\varphi}\\ 
&\qquad\qquad\qquad+\log\log\varphi\left(\frac{c_3}{\log\varphi}-1\right)+c_2+2\gamma+\psi_1\hspace{-.1cm}\left(\frac{\log n}{\log\varphi}\right)\Bigg),\\
a(n)&:=\frac{1}{2\log\varphi}\left(\log n-2\log\log n-\frac{4\psi_0\hspace{-.1cm}\left(\frac{\log n}{\log\varphi}\right)}{\log\varphi}+2\log\log\varphi+2c_3-2\log\varphi\right)\\
b(n)&:=\frac{1}{2\log\varphi}\left(\log\log n+2-\frac{4\psi_0\hspace{-.1cm}\left(\frac{\log n}{\log\varphi}\right)}{\log\varphi}+2\log\log\varphi-2c_3+2\log\varphi\right).\end{align*} Combing the asymptotic for $F_2(e^{-\alpha})$ with equations \eqref{eq:T1} and \eqref{eq:T2} gives the desired result.
\end{proof}

\section{Partitions over general linear recurrences with a dominant root}\label{sec:P}

The results of the previous sections on Fibonacci partitions can be generalised to positive recurrence sequences $P_n$ with an irreducible characteristic polynomial having a positive dominant real root and such that $P_1=1$. We will also assume that the values $P_n$ are distinct. Here, we require $P_1=1$ so that there always exists a non-distinct partition of $n$ over $P_n$. The restriction on the values of $P_n$ being distinct is not a very strict one---since the recursion has a dominant real root there are at most finitely many repeated values, so an analysis analogous to moving between $F(x)$ and $F_2(x)$ is possible, and uncomplicated. For such $P_n$, we wish to asymptotically understand the number of solutions to \begin{equation}\label{eq:nP} n=a_1P_1+a_2P_2+\cdots+a_kP_k+\cdots.\end{equation} As before, we let $p_P(n)$ denote the number of non-distinct partitions of $n$ over the sequence $P_n$, that is, solutions to \eqref{eq:nfib} in nonnegative integers $a_k$.

The case of partitions over $P_n$ is carried out exactly as in the case with the Fibonacci numbers $F_n$, but with the Fibonacci zeta function replaced by the zeta function $\zeta_P(z)$ which is the meromorphic continuation of the Dirichlet series $\sum_{k\geqslant 1}P_k^{-z}$ ($\Re(z)>0$). Here, we consider the generating function, $$F_P(x):=\sum_{n\geqslant 0}p_P(n)x^n=\prod_{k\geqslant 1}\left(1-x^{P_k}\right)^{-1}.$$ 

To complete our analysis, we use the following result of Serrano Holdago and Navas Vicente~\cite{HN2023}, which is a generalisation of Navas \cite{N2001}.

\begin{proposition}[Serrano Holdago and Navas Vicente, 2023]\label{prop:HN2023} Let $P(x)$ be the minimal polynomial with $\deg P(x)=r$ of the linear recurrence $P_n$ (as described above) and let $\beta>1$ be the dominant root of $P(x)$. Then the Dirichlet series $\zeta_P(z):=\sum_{k\geqslant 1}P_k^{-z}$ ($\Re(z)>0$) can be analytically continued to a meromorphic function, also denoted $\zeta_P(z)$, all of whose singularities are simple poles at the points $$s{(n,\boldsymbol{k})}:=\frac{\log|\beta^{-k_1}\beta_2^{k_1-k_2}\cdots \beta_r^{k_{r-1}}|}{\log \beta}+i\cdot\frac{\arg(\beta_1^{-k_1}\beta_2^{k_1-k_2}\cdots \beta_r^{k_{r-1}})+2\pi n}{\log \beta},$$ where $\boldsymbol{k}=(k_1,\ldots,k_{r-1})$, $\beta_2,\ldots,\beta_r$ are the algebraic conjugates of $\beta$, $n\in\mathbb{Z}$ and the parameters $k_1,\ldots,k_{r-1}$ are integers satisfying $0\leqslant k_{r-1}\leqslant k_{r-2}\leqslant \cdots \leqslant k_{1}$.
\end{proposition}

We adopt the terminology of Proposition \ref{prop:HN2023} for the rest of this section along with the definitions of the real numbers $\lambda=\lambda_1,\lambda_2,\ldots,\lambda_r$, which satisfy $$P_n=\lambda\beta^n+\lambda_2\beta_2^n+\cdots+\lambda_r\beta_r^n.$$ Note that since $P_n$ is strictly increasing, we necessarily have that $\lambda>0$.

Continuing the analogy with the Fibonacci partitions, we need asymptotic results, as $s\to0^+$, of the functions $$\log F_P(e^{-s})= \frac{1}{2\pi i} \int_{a-i\infty}^{a+i\infty} s^{-z} \Gamma(z) \zeta(1+z)\zeta_P(z) dz.$$ As before, we have a few different types of poles to consider:
\begin{itemize}
\item a triple pole at $z=0$ in the case of $F_P(e^{-s})$,
\item simple poles at countable (and separated) non-integer real values $z\leqslant -\frac{\log(\beta/|\beta_2|)}{\log\beta}.$
\item double poles at $z\in-\mathbb{N}$, and
\item simple poles off the real line at $z=s(n,\boldsymbol{k})$.
\end{itemize}
Now, it is clear from the previous analysis on Fibonacci partitions that the main contribution will come from the pole at $z=0$, which comes only from $\boldsymbol{k}=\boldsymbol{0}=(0,\ldots,0).$ The negative real poles will give a cumulative contribution of $O\big(s^{\min\{{\log(\beta/|\beta_2|)}/{\log\beta},1-\varepsilon\}})$ for any fixed $\varepsilon>0$, since $s\log s=O(s^{1-\varepsilon})$ for any fixed $\varepsilon>0$ as $s\to0^+$. As well, the simple poles off the real line contribute a function $f_2(s)$ toward $\log F_P(e^{-s})$, that satisfies $f_2(s)=f_2(\beta s)+O\big(s^{\min\{{\log(\beta/|\beta_2|)}/{\log\beta},1-\varepsilon\}})$. It remains to obtain the contributions from the pole at $z=0$. For these, we require the following result. 

\begin{lemma}\label{lem:zetaP_0} Near $z=0$, we have $$\zeta_P(z)=\frac{1}{\log\beta}\cdot\frac{1}{z}-\frac{\log\lambda}{\log\beta}-\frac{1}{2}+\left(\frac{(\log\lambda)^2}{2\log\beta}+\frac{\log\lambda}{2}+\frac{\log\beta}{12}+C_1\right)z+O(z^2),$$ where $C_1:=\sum_{k\geqslant 1}\sum_{\boldsymbol{k}\neq\boldsymbol{0}}\frac{1}{k_1}{k_1\choose k_2}\cdots{k_{r-2}\choose k_{r-1}}\frac{(\lambda_{r}\beta_{r}^k)^{k_{r-1}}}{(\lambda\beta^k)^{z+k_{1}}}\left(\prod_{j=2}^{r-1}(\lambda_{j}\beta_{j}^k)^{k_{j+1}-k_{j}}\right)$.
\end{lemma}

\begin{proof} Following Serrano Holdago and Navas Vicente \cite{HN2023}, we write \begin{align*}\zeta_P(z)&=\sum_{k\geqslant 1}\sum_{\boldsymbol{k}}{-z\choose k_1}{k_1\choose k_2}\cdots{k_{r-2}\choose k_{r-1}}\frac{(\lambda_{r}\beta_{r}^k)^{k_{r-1}}}{(\lambda\beta^k)^{z+k_{1}}}\left(\prod_{j=2}^{r-1}(\lambda_{j}\beta_{j}^k)^{k_{j+1}-k_{j}}\right)\\
&=\sum_{k\geqslant 1}\frac{1}{(\lambda\beta^k)^{z}}+\sum_{k\geqslant 1}\sum_{\boldsymbol{k}\neq\boldsymbol{0}}{-z\choose k_1}{k_1\choose k_2}\cdots{k_{r-2}\choose k_{r-1}}\frac{(\lambda_{r}\beta_{r}^k)^{k_{r-1}}}{(\lambda\beta^k)^{z+k_{1}}}\left(\prod_{j=2}^{r-1}(\lambda_{j}\beta_{j}^k)^{k_{j+1}-k_{j}}\right)\\
&=\frac{\lambda^{-z}}{\beta^z-1}+\sum_{k\geqslant 1}\sum_{\boldsymbol{k}\neq\boldsymbol{0}}\frac{\Gamma(z+k_1)}{\Gamma(z)\, k_1!}{k_1\choose k_2}\cdots{k_{r-2}\choose k_{r-1}}\frac{(\lambda_{r}\beta_{r}^k)^{k_{r-1}}}{(\lambda\beta^k)^{z+k_{1}}}\left(\prod_{j=2}^{r-1}(\lambda_{j}\beta_{j}^k)^{k_{j+1}-k_{j}}\right). 
\end{align*} where $\sum_{\boldsymbol{k}}:=\sum_{k_1\geqslant 0}\sum_{k_2=0}^{k_1}\cdots\sum_{k_{r-1}=0}^{k_{r-2}}.$ Again, using that $\tfrac{1}{\Gamma(z)} = z + O(z^2)$ and $\Gamma(z+k_1)= (k_1-1)! + O(z)$ near $z=0$, along with the asymptotic expansions $$\lambda^{-z}=1-z\log\lambda+\frac{1}{2}(\log\lambda)^2z^2+O(z^3),$$ and $$\frac{1}{\beta^z-1}=\frac{1}{\log\beta}\cdot\frac{1}{z} -\frac{1}{2}  + \frac{\log\beta}{12}z + O(z^2),$$ we have that $$\zeta_P(z)=\frac{1}{\log\beta}\cdot\frac{1}{z}-\frac{\log\lambda}{\log\beta}-\frac{1}{2}+\left(\frac{(\log\lambda)^2}{2\log\beta}+\frac{\log\lambda}{2}+\frac{\log\beta}{12}+C_1\right)z+O(z^2).\qedhere$$\end{proof}

In the following result, we use Lemma \ref{lem:zetaP_0} to determine the asymptotic behaviour of the function $\log F_P(e^{-s})$ as $s\to 0^+$. All of the expansions of the functions involved, $s^{-z}$, $\Gamma(z)$ and $\zeta(1+z)$, have been noted somewhere in the previous sections of this work---we use them below without further reference.

\begin{proposition}\label{prop:logFPGP}
For any fixed $\varepsilon\in(0,1)$, as $s\to0^+$, the function $F_P(z)$  defined above satisfies, $$\log F_P(e^{-s})=\frac{(\log s)^2}{2\log\beta}+\left(\frac{\log\lambda}{\log\beta}+\frac{1}{2}-\frac{2\gamma}{\log\beta}\right)\log s+C_2+f_2(s)+O\big(s^{\min\{{\log(\beta/|\beta_2|)}/{\log\beta},1-\varepsilon\}}),$$ where $$C_2:= \frac{1}{\log\beta}\left(\frac{1}{2}\left(\gamma^2+\frac{\pi^2}{6}\right)+\gamma^2-\gamma_1\right)-\left(\frac{\log\lambda}{\log\beta}+\frac{1}{2}\right)2\gamma+\left(\frac{(\log\lambda)^2}{2\log\beta}+\frac{\log\lambda}{2}+\frac{\log\beta}{12}+C_1\right),$$ and $f_2(s)=f_2(\beta s)+O\big(s^{\min\{{\log(\beta/|\beta_2|)}/{\log\beta},1-\varepsilon\}})$.
\end{proposition}

\begin{proof} Note that near $z=0$, we have \begin{align*}s^{-z}\Gamma(z)\zeta(1+z)=\frac{1}{z^2}+(2\gamma-&\log s)\frac{1}{z}\\ &+\left(\frac{1}{2}\left(\gamma^2+\frac{\pi^2}{6}\right)+\gamma^2-\gamma_1-2\gamma\log s+\frac{(\log s)^2}{2}\right)+O(z),\end{align*} so that, using Lemma \ref{lem:zetaP_0}, we have that \begin{align*}\res{z}{0}{s^{-z}\Gamma(z)\zeta(1+z)\zeta_P(z)}&=\frac{1}{\log\beta}\left(\frac{1}{2}\left(\gamma^2+\frac{\pi^2}{6}\right)+\gamma^2-\gamma_1-2\gamma\log s+\frac{(\log s)^2}{2}\right)\\
&\qquad-\left(\frac{\log\lambda}{\log\beta}+\frac{1}{2}\right)(2\gamma-\log s)\\
&\qquad\qquad+\left(\frac{(\log\lambda)^2}{2\log\beta}+\frac{\log\lambda}{2}+\frac{\log\beta}{12}+C_1\right)\\
&=\frac{(\log s)^2}{2\log\beta}+\left(\frac{\log\lambda}{\log\beta}+\frac{1}{2}-\frac{2\gamma}{\log\beta}\right)\log s+C_2.
\end{align*} The contributions coming from the rest of the poles are as discussed before the statement of Lemma~\ref{lem:zetaP_0}.
\end{proof}

On inspection, one notices that the dominant asymptotic terms of $\log F_P(e^{-s})$ are precisely the dominant asymptotic terms of $\log F_2(e^{-s})$, the function related to the Fibonacci partitions, after substituting $\beta$ for $\varphi$. Of course, this is not so surprising, as the Fibonacci numbers $F_k$ are just a special case of the more general sequence $P_k$. The property of note, here, is regarding the associated saddle point. Since the leading order behaviour is the same, the saddle points satisfy the same leading order asymptotics. In particular, Proposition~\ref{prop:alphan} and its corollaries hold for the saddle point $\alpha_{F_P}$ related to non-distinct partitions of $n$ over $P_k$, and so also then, do \eqref{eq:T1} and \eqref{eq:T2}. We use these results, as well as this notation, below. 

\begin{proof}[Proof of Theorem \ref{thm:genpart}] By Proposition~\ref{prop:logFPGP}, Proposition~\ref{prop:alphan} and its corollaries, the small-$\alpha_{F_P}$ asymptotics, or, equivalently, the large-$n$ asymptotics satisfy, \begin{align*}\log F_P(e^{-\alpha_{F_p}})&=\frac{(\log \alpha_{F_p})^2}{2\log\beta}-(\log \alpha_{F_p})\left(\frac{2\gamma}{\log\beta}-\frac{\log\lambda}{\log\beta}-\frac{1}{2}\right)\\
&\qquad\qquad+C_2+f_2(\alpha_{F_p})+O\Big(\alpha_{F_p}^{\min\{{\log(\beta/|\beta_2|)}/{\log\beta},1-\varepsilon\}}\Big)\\
&=\frac{\log n}{2\log\beta}\left(\log n-2\log\log n-\frac{4\psi_3\hspace{-.1cm}\left(\frac{\log n}{\log\varphi}\right)}{\log\beta}+2\log\log\beta+4\gamma -2\log\lambda-\log\beta\right)\\
&\qquad+\frac{\log\log n}{2\log\beta}\left(\log\log n+2-\frac{4\psi_3\hspace{-.1cm}\left(\frac{\log n}{\log\beta}\right)}{\log\beta}+2\log\log\beta-4\gamma +2\log\lambda+\log\beta\right)\\
&\qquad\qquad+\frac{(\log\log\beta)^2}{2\log\beta}-\frac{\psi_3\hspace{-.1cm}\left(\frac{\log n}{\log\beta}\right)}{(\log\beta)^2}+(1+2\gamma-\log\lambda)\frac{\log\log\beta}{\log\beta}-\frac{1}{2}\log\log\beta\\
&\qquad\qquad\qquad+C_2+\psi_4\hspace{-.1cm}\left(\frac{\log n}{\log\beta}\right)+O\left(\frac{(\log\log n)^2}{\log n}\right),\end{align*}
where $\psi_3(x)$ and $\psi_4(x)$ are explicitly computable $1$-periodic functions that are analogous to $\psi_0(x)$ and $\psi_1(x)$, respectively. Thus, we have \begin{equation*}F_P(e^{-\alpha_{F_P}})=\psi_5\hspace{-.1cm}\left(\frac{\log n}{\log\beta}\right) n^{c(n)}(\log n)^{d(n)}\left(1+O\left(\frac{(\log\log n)^2}{\log n}\right)\right),\end{equation*} where $\psi_5(n)$ is the $1$-periodic positive function \begin{align*}\psi_5(x)&:=\exp\Bigg(\frac{(\log\log\beta)^2}{2\log\beta}-\frac{\psi_3(x)}{(\log\beta)^2}+(1+2\gamma-\log\lambda)\frac{\log\log\beta}{\log\beta}-\frac{1}{2}\log\log\beta+C_2+\psi_4(x)\Bigg),\\
c(n)&:=\frac{1}{2\log\beta}\left(\log n-2\log\log n-\frac{4\psi_3\hspace{-.1cm}\left(\frac{\log n}{\log\varphi}\right)}{\log\beta}+2\log\log\beta+4\gamma -2\log\lambda-\log\beta\right),\\
d(n)&:=\frac{1}{2\log\beta}\left(\log\log n+2-\frac{4\psi_3\hspace{-.1cm}\left(\frac{\log n}{\log\beta}\right)}{\log\beta}+2\log\log\beta-4\gamma +2\log\lambda+\log\beta\right).\end{align*} Combing the asymptotic for $F_P(e^{-\alpha_{F_P}})$ with equations \eqref{eq:T1} and \eqref{eq:T2} gives the desired result.
\end{proof}

\subsection*{Acknowledgements}
This work was started during a short visit of M.~Coons to Aarhus University; he thanks their faculty and staff for their hospitality and support. S.~Kristensen and M.~L.~Laursen are supported by the Independent Research Fund Denmark (Grant ref. 1026-00081B) and the Aarhus University Research Foundation (Grant ref. AUFF-E-2021-9-20), and M.~Coons is supported by the David W.~and Helen E.~F.~Lantis Endowment. The idea for this work came about when M.~Coons was a master's student working under Klaus Kirsten---this work benefitted from several conversations with Klaus Kirsten and also Markus Hunziker, which occurred nearly twenty years ago! M.~Coons gives a hearty, and very belated, thanks to both of them.

\bibliographystyle{amsplain}

\begin{thebibliography}{1}

\bibitem{dB1948}
N.~G. de~Bruijn, \emph{On {M}ahler's partition problem}, Indagationes Math.~\textbf{10} (1948), 210--220.

\bibitem{CS2021}
S.~Chow and T.~Slattery, \emph{On Fibonacci partitions}, J. Number Theory \textbf{225} (2021), 310--326.

\bibitem{CK2009}
M.~Coons and K.~Kirsten, \emph{General moment theorems for nondistinct unrestricted partitions}, J.~Math.~Phys.~\textbf{50} (2009), Article 013517.

\bibitem{DT2020}
G.~Debruyne and G.~Tenenbaum, \emph{The saddle-point method for general partition functions}, Indagationes Math.~\textbf{31} (2020), 728--738.

\bibitem{C1996} R.~M.~Corless, G.~H.~Gonnet, D.~E.~G.~Hare, D.~J.~Jeffrey, and D.~E.~Knuth, \emph{On the Lambert $W$ function}, Adv. Comput. Math. \textbf{5} (1996), no.4, 329--359.

\bibitem{HR1918}
G.~H.~Hardy and S.~Ramanujan, \emph{Asymptotic formul\ae\ in combinatorial analysis}, Proc.~Lond.~Math.~Soc.~(2) \textbf{17} (1918), 75--115.

\bibitem{Kpre}
T.~Kempton, \emph{The dynamics of the Fibonacci partition function}, preprint. {\tt arXiv:2311.06006}

\bibitem{M1940}
K.~Mahler, \emph{On a special functional equation}, J.~London Math.~Soc.~\textbf{15} (1940), 115--123. 

\bibitem{N1954}
V.~S.~Nanda, \emph{Partition theory and thermodynamics
of multidimensional oscillator assemblies}, Proc. Camb. Phil. Soc. \textbf{47} (1954), 591--601.

\bibitem{N2001}
L.~Navas, \emph{Analytic continuation of the {Fibonacci} {Dirichlet} series},
  Fibonacci Q. \textbf{39} (2001), no.~5, 409--418.
  
\bibitem{O1974} 
F.~W.~J.~Olver, {Asymptotics and Special Functions},
Academic Press, New York, 1974.

\bibitem{R1975}
L.~B.~Richmond, \emph{The moments of partitions I},
Acta Arith.~\textbf{26} (1975), 411--425.

\bibitem{R1975/76}
L.~B.~Richmond, \emph{The moments of partitions II},
Acta Arith.~\textbf{28} (1975/76), 229--243.

\bibitem{CSpre}
C.~Sana, \emph{A note on the power sums of the number of Fibonacci partitions}, preprint. {\tt arXiv:2309.12724}.

\bibitem{HN2023}
\'{A}.~Serrano Holgado and L.~M.~Navas Vicente, \emph{The zeta function of a recurrence sequence of arbitrary degree}, Mediterr. J. Math. \textbf{20} (2023), no.~4, Paper No.~224, 18 pp.

\bibitem{T1986}
E.~C.~Titchmarsh, The theory of the Riemann zeta-function, 2 ed., Oxford University Press, New York, 1986.

\end{thebibliography}
\providecommand{\bysame}{\leavevmode\hbox to3em{\hrulefill}\thinspace}
\providecommand{\href}[2]{#2}


\end{document}